\documentclass{amsart}

\usepackage{amssymb}
\usepackage{bm}
\usepackage{graphicx}
\usepackage{psfrag}
\usepackage{color}
\usepackage{url}
\usepackage{amsmath,color}
\usepackage{enumerate}
\usepackage{algpseudocode}

\usepackage{xy}
\input xy
\xyoption{all}

\newcommand{\excise}[1]{}

\newcommand{\ve}{\boldsymbol}

\newtheorem{thm}{Theorem}
\newtheorem{lemma}[thm]{Lemma}
\newtheorem{cor}[thm]{Corollary}
\newtheorem{prop}[thm]{Proposition}

\theoremstyle{definition}

\newtheorem{example}[thm]{Example}{\bfseries}{\normalfont}
\newtheorem{remark}[thm]{Remark}{\bfseries}{\normalfont}

\newtheorem{defn}[thm]{Definition}{\bfseries}{\normalfont}

\newcommand{\ring}[1]{\ensuremath{\mathbb{#1}}}

\newcommand{\spn}{\mathrm{span}}
\newcommand{\rank}{\mathrm{rank}}
 \newcommand{\R}{\mathbb R}
 \newcommand{\Z}{\mathbb Z}

\newcommand{\sg}{\mbox{\rm Sg}}

\renewcommand\>{\rangle}
\newcommand\<{\langle}

\newcommand\QQ{\ring{Q}}
\newcommand\RR{\ring{R}}

\newcommand\ZZ{\ring{Z}}

\newcommand\kk{\Bbbk}

\newcommand\ee{{\mathbf e}}
\newcommand\vv[1]{{\mathbf{#1}}}

\newcommand\supp{\mathrm{supp}}
\newcommand\cone{\mathrm{cone}}
\newcommand\icone{\mathrm{Sg}}
\newcommand\lcm{\mathrm{lcm}}

\newcommand{\modulo}{\,\mathrm{mod}\;}

\newcommand\iso{\cong}

\begin{document}

\title[Sparse solutions of linear Diophantine equations]{Sparse Solutions of Linear Diophantine Equations}

\author{Iskander Aliev}
\address{Cardiff University, UK}
\email{alievi@cardiff.ac.uk}

\author{Jes\'us A. De Loera}
\address{University of California, Davis, USA}
\email{deloera@math.ucdavis.edu}

\author{Timm Oertel}
\address{Cardiff University, UK}
\email{oertelt@cardiff.ac.uk}

\author{Christopher O'Neill}
\address{University of California, Davis, USA}
\email{coneill@math.ucdavis.edu}



\date{\today}

\begin{abstract}
\hspace{-2.05032pt}
We present structural results on solutions to the Diophantine system
$A{\ve y} = {\ve b}$, ${\ve y} \in \ZZ^t_{\ge 0}$
with the smallest number of non-zero entries.  Our tools are algebraic and number theoretic in nature and include Siegel's Lemma, generating functions, and commutative algebra.  These results have some interesting consequences in discrete optimization.
\end{abstract}


\maketitle


\section{Introduction}\label{s:intro}

Let $A$ be an integer $d\times t$ matrix and let ${\ve b}$ be an integer $d$-dimensional vector.
The purpose of this work is to study structural properties of the solutions to the non-linear integer optimization problem
\begin{equation}\label{eq:MainProblem}
\min \{\| {\ve y} \|_0: \, A{\ve y} = {\ve b},  {\ve y} \in \ZZ^t_{\ge 0}\}.%
\end{equation}

Here, $\|\cdot\|_0$ denotes the $0$-norm, which counts the cardinality of the \emph{support} of ${\ve y}$, i.e.\ $\supp({\ve y})=\{i : y_i \not= 0\}$.  In other words, the value of $\|{\ve y} \|_0$ equals the number of non-zero entries in the vector ${\ve y}$.  Problem (\ref{eq:MainProblem}) aims to find the vector of minimal~support.

Before we present our results and state prior work we introduce some basic notation. In this paper, $\log(a)$ refers to the base two logarithm used
to measure bit-size.
Let $A = ({\ve x}_1, \ldots, {\ve x}_t )\in \ZZ^{d\times t}$ be a matrix (defining Problem (\ref{eq:MainProblem})), where the columns come from a finite set of vectors $X=\{{\ve x}_1, \ldots, {\ve x}_t \}$. In what follows
we sometimes use $A$ and $X$ interchangeably.

The \emph{conic hull} of $X$ is the set
\begin{equation*}
\cone(X) = \{\lambda_1 {\ve x}_1+ \cdots +\lambda_t {\ve x}_t:  {\ve x}_1, \ldots, {\ve x}_t \in X, \lambda_1, \ldots, \lambda_t \in \RR_{\ge 0}\},
\end{equation*}
and the \emph{semigroup} of $X$ or the \emph{integer conic hull} of $X$ is the set
\begin{equation*}
\icone(X) = \{\lambda_1 {\ve x}_1+ \cdots +\lambda_t {\ve x}_t: {\ve x}_1, \ldots, {\ve x}_t \in X, \lambda_1, \ldots, \lambda_t \in \ZZ_{\ge 0} \}.
\end{equation*}
For each ${\ve b} \in \icone(X)$, we let
\begin{equation*}
\mathsf P_X({\ve b}) = \{{\ve \lambda} \in \ZZ_{\ge 0}^t : \lambda_1{\ve x}_1 + \cdots + \lambda_t{\ve x}_t = {\ve b}\}
\end{equation*}
denote the solution set for ${\ve b}$.
We are interested in 
 the asymptotic behavior of the function
\begin{equation*}
\mathsf m_0({\ve b}) = \min\{ \| {\ve \lambda} \|_0 : {\ve \lambda} \in \mathsf P_X({\ve b})  \},
\end{equation*}
as well as methods of estimating the function
\begin{equation*}
{\mathsf M}_0(X) = \max\{{\mathsf m}_0({\ve b}): {\ve b} \in \icone(X)\}\,.
\end{equation*}

Finding the sparsest solution of a system of linear equations has many applications and there is a rich literature about this problem.  For real variables, the $0$-norm minimization problem has become quite popular in signal processing through the theory of \emph{compressed sensing}.  It is known that a linear programming relaxation provides a guaranteed approximation \cite{candesetal2006stable, candestao}. Moreover, many nice properties for the size of the solution are known for the case of random matrices \cite{brucksteinetal,candestao}.

In this paper, we consider \emph{integer} solutions, which have various applications as well.  Integer sparsity
appears in the setting of linear codes over finite fields, where the $0$-norm is the \emph{Hamming distance}, and the problem is closely related to the \emph{nearest codeword problem} as well as to the problem of finding shortest cycles on graphs and matroids; see \cite{alonetal,choetal,micciancio1,vardy} and references therein.  Sparse integer solutions also appear in the context of finding guarantees for bin-packing problems via the Gilmore-Gomory formulation \cite{gilmore+gomory}, as first suggested in \cite{karmakar+karp}.  More generally, upper bounds given for the size of the sparsest integer solution indicate that if there exists an optimal solution to such an integer program, then there exists one which is polynomial in the number of equations and the maximum binary encoding length among integers in the objective function vector and the constraint matrix (see \cite[Section~3]{EisenbrandShmonin2006}).  The problem of estimating ${\mathsf M}_0(X)$ goes back to classical results on the integer Carath\'eodory problem.  Cook, Fonlupt, and Schrijver \cite{CookCaratheodory} showed that ${\mathsf M}_0(X) \le 2d-1$ if $C = \cone(X)$ is pointed and $X$ forms a Hilbert basis of $C$.  This result was later improved by Seb\H{o} \cite{SeboCaratheodory} to ${\mathsf M}_0(X)\le 2d-2$. It remains an open question to determine a sharp upper bound in the Hilbert basis setting, in \cite{brunsetal} an example is provided where ${\mathsf M}_0(X)= \lfloor \frac{7}{6} d \rfloor$.
For an arbitrary finite set $X \subset \ZZ^d$, Eisenbrand and Shmonin \cite{EisenbrandShmonin2006} obtained the bound
\begin{equation}\label{eq:ESbound}
{\mathsf M}_0(X) \le 2d \log (4d\| X \|_{\infty}),
\end{equation}
where $\| X \|_{\infty} = \max_{{\ve x} \in X} \| {\ve x} \|_{\infty}$ (recall that for a vector ${\ve x}$,  one defines $\|{\ve x} \|_{\infty}=\max |x_i|$).

For linearly independent vectors ${\ve v}_1, \ldots, {\ve v}_r$ in $\R^t$, the set $\Lambda=\{\sum_{i=1}^{r} z_i {\ve v}_i,\, z_i\in \Z\}$ is a $r$-dimensional {\em lattice} with {\em basis} ${\ve v}_1, \ldots, {\ve v}_r$ and {\em determinant}
$\det(\Lambda)=(\det({\ve v}_i\cdot {\ve v}_j)_{1\le i,j\le r})^{1/2}$, where ${\ve v}_i\cdot {\ve v}_j$ is the standard inner product of the basis vectors ${\ve v}_i$ and  ${\ve v}_j$.
In what follows, we denote $X = \{{\ve x}_1,\ldots, {\ve x}_t\} \subset \ZZ^d$, and let $W=W(X)$ be the $t\times d$ matrix with rows ${\ve x}_1^T, \ldots, {\ve x}_t^T$ and let $r = r(X)$ be the column rank of the matrix $W$. We will denote by $\Lambda(X)$ the $r$-dimensional sublattice of $\ZZ^t$ formed by all integer points in the linear subspace $\spn_{\R}(W)$ spanned by the columns of $W$, that is
\begin{equation}
\Lambda(X)=\spn_{\R}(W)\cap \Z^t\,.
\end{equation}
Similar, we let
\begin{equation*}
H(X) = \det(\Lambda(X))
\end{equation*}
be the determinant of the lattice $\Lambda(X)$. Note that
\begin{equation}\label{SmithExpr}
H(X) = g^{-1}\sqrt{{\rm det}(V^TV)}\,,
\end{equation}
where  $V$ is a matrix formed by any $r$ linearly independent columns of $W$ and $g$ is
the greatest common divisor of the determinants of all submatrices of $V$ of order $r$
(see \cite[Chapter~1, \S 1]{Skolem} and \cite{Smith}).

\subsection{Our contributions}

In this paper, we study two questions:

\begin{enumerate}
\item
What are the best bounds we can give for $\mathsf {\mathsf M}_0(X)$ in terms of the generating set $X$?

Our first main result is Theorem \ref{t:Bound_via_determinant}. There, for general $X$, we obtain two new upper bounds for ${\mathsf M}_0(X)$ that improve upon (\ref{eq:ESbound}) in two distinct ways.

\begin{thm}\label{t:Bound_via_determinant}
Let $X \subset \ZZ^d$ be a finite set of nonzero integer vectors.  Then
\begin{enumerate}[(i)]

\item
\label{t:Bound_via_determinant_bound1}
${\mathsf M}_0(X) \le r(X) + \lfloor\log(H(X))\rfloor$, and

\item
\label{t:Bound_via_determinant_bound2}
${\mathsf M}_0(X) \le \lfloor 2d\log(2\sqrt{d}||X||_{\infty})\rfloor$.

\end{enumerate}
\end{thm}

It should be pointed out here that the bound in Theorem~\ref{t:Bound_via_determinant} (i) depends only on the rank and $H(X)$, the height of a rational subspace $S$ spanned by $W(X)$ and remains the same for any finite set of nonzero integer vectors $X'$ with $S=\spn_{\R}(W(X'))$.  Hence the bound in Theorem~\ref{t:Bound_via_determinant} (i) becomes arbitrarily smaller than the bound (\ref{eq:ESbound}) provided $\|X\|_{\infty}$ is sufficiently large. 

The second result is Theorem~\ref{t:Lemma_Positive_Knapsack_refinement} below, which refines Theorem~\ref{t:Bound_via_determinant} for knapsack problems with positive entries (that is, when $X \subset \ZZ_{> 0}$ is a list of positive integers).

\begin{thm}\label{t:Lemma_Positive_Knapsack_refinement}
Let $X\subset \Z_{>0}$ be a finite set of positive integers. Then
\begin{equation}\label{eq:Positive_Knapsack_Bound_via_determinant}
{\mathsf M}_0(X) \leq 1 + \lfloor \log(\|X\|_{\infty})\rfloor\,.
\end{equation}
\end{thm}

Both theorems are proved in Section~\ref{s:caratheodory}.

\item
What is the asymptotic behavior of the univariate function $ \mathsf m_0(\lambda {\ve b})$ obtained from successive dilations of the vector ${\ve b}$?

We prove the somewhat surprising result that this function of $\lambda$ is eventually a \emph{periodic} function.  The precise statement (Theorem~\ref{t:eventualm0affine}) is reminiscent of the behavior of Ehrhart and volume functions of polyhedra; see \cite{BarviPom} for more on polyhedral combinatorics.
We also present Theorem~\ref{t:eventualm0knap}, which is a much sharper result for knapsacks problems.  The details of these results are discussed in Section~\ref{s:multiquasi}.

\end{enumerate}

In order to stress the applicability of our results in discrete optimization, we conclude this introduction with the following immediate corollary of Theorem~\ref{t:Bound_via_determinant}, which gives an interesting bound on the sparsity of optimal solutions in integer and mixed integer linear programs.

\begin{cor}\label{c:Bound_via_determinant}
Fix a matrix $A \in \Z^{d \times t}$ with columns ${\ve a}_1, \ldots, {\ve a}_t$ and objective function ${\ve c} \in \Z^t$.  If the integer program
$$\min \{{\ve c}^T {\ve x} :  A{\ve x}={\ve b}, \ {\ve x}\geq 0, \  {\ve x}\in \Z^t\}$$
has a finite optimum, then there exists an optimal solution ${\ve x}^*$ with
at most
$$(d+1) + \lfloor\log(H(X))\rfloor$$
non-zero components, where $X$ is given by the enlarged column vectors
$$\left\{\left(\genfrac{}{}{0pt}{}{{\ve a}_1}{c_1} \right), \ldots, \left(\genfrac{}{}{0pt}{}{{\ve a}_1}{c_1}\right)\right\}.$$

More generally, if  $A$ and $B$ are $d \times t_1$, $d \times t_2$ matrices, and an optimum exists for the mixed integer program
$$\min \{{\ve c}^T {\ve x} + {\ve v}^T {\ve y} :  A{\ve x} + B {\ve y} = {\ve b}, \ {\ve x}\geq 0, \ {\ve y}\ge0, \  {\ve x}\in \Z^{t_1}, \ {\ve y} \in \R^{t_2} \},$$
then there is an optimal solution
with at most
$$(d + 1) + \lfloor\log(H(X))\rfloor + \rank(B)$$
non-zero components, where $X$ is as before.
\end{cor}

\section{Bounds for sparsity in solutions through Siegel's Lemma}
\label{s:caratheodory}

The proof of \eqref{eq:ESbound}, and some of the proofs in \cite{SeboCaratheodory}, make use of equal sub-sums of the set of vectors $X$ to decrease the number of elements needed to represent a given vector ${\ve b}\in \sg(X)$.  In this section, we use Siegel's Lemma and the geometry of numbers to refine \eqref{eq:ESbound}. We now review some useful results.

\subsection{An application of Siegel's Lemma}%

Given an integer matrix $A \in \ZZ^{m \times n}$, $m < n$, with $m$ linearly independent rows, the system $A{\ve y} = {\ve 0}$ has a nontrivial integer solution.  If the coefficients are small integers, then there will be a solution in small integers.  Thue was the first to use this principle in \cite{Th}, Siegel was the first to state this idea formally (\cite{Siegel}, Bd. I, p. 213, Hilfssatz).   Bombieri and Vaaler \cite{BombVaal} obtained the following general~result.

\begin{thm}[Bombieri and Vaaler, 1983]\label{t:SiegelsLemma}
Fix an $m \times n$ integer matrix $A$, $m < n$, with $m$ linearly independent rows.  To the system of equations $A{\ve y} = {\ve 0}$, there are $n - m$ linearly independent integral solutions ${\ve y}_1, \ldots,{\ve y}_{n-m} $ satisfying
\begin{eqnarray}\label{eq:sl_sl_f}
\prod_{l=1}^{n-m} || {\ve y}_l||_{\infty} \le g^{-1}\sqrt{{\rm det}(AA^{T})}\,,
\end{eqnarray}
where $g$ is the greatest common divisor of the determinants of all $m \times m$ submatrices of~$A$.
\end{thm}
Note that (\ref{eq:sl_sl_f}) is optimal up to a constant multiple depending only on $n$ and $m$~\cite{VaalerBest}.
Siegel's Lemma plays the key role in the proof of the following result.

\begin{lemma}\label{t:Lemma_ES_refinement}
Let $X \subset \ZZ^d$ be a finite set of nonzero integer vectors and let ${\ve b} \in \icone(X)$. If
\begin{equation}\label{eq:Bound_via_determinant}
|X| > r(X) + \log(H(X))\,
\end{equation}
then there exists a proper subset $Y\subset X$ such that ${\ve b} \in \icone(Y)$.
\end{lemma}

\begin{proof}[Proof of Lemma~\ref{t:Lemma_ES_refinement}]
We will use Theorem~\ref{t:SiegelsLemma}.  Suppose that
\begin{equation}\label{eq:opposite_inequality}
|X| = t > r(X) + \log(H(X))\,.
\end{equation}
We need to show that there exists a proper subset $Y$ of $X$
such that ${\ve b} \in \icone(Y)$.

The inequality (\ref{eq:opposite_inequality}) implies
\begin{equation*}
H(X)<2^{t-r}\,,
\end{equation*}
where $r=r(X)$.
Let $V$ be a matrix formed by $r$ linearly independent columns of the matrix $W(X)$. Then $H(X) = g^{-1}\sqrt{{\rm det}(V^TV)}$, where $g$ is the greatest common divisor of the determinants of all submatrices of $V$ of order $r$.  Applying Theorem~\ref{t:SiegelsLemma} to the matrix $V^T$, there exists a nonzero vector ${\ve y} \in \ZZ^t$ such that
\begin{equation}\label{InKernel}
V^T{\ve y}={\ve 0}
\end{equation}
and
%
\begin{equation*}||{\ve y}||_\infty \le H(X)^{1/(t-r)} < 2\,,
\end{equation*}
%
so ${\ve y} = (y_1, \ldots, y_t)^T$ has entries $y_i \in \{-1,0,1\}$.
By (\ref{InKernel}) we also have $W(X)^T{\ve y}={\ve 0}$, so that
\begin{equation*}
y_1 {\ve x}_1 + \cdots + y_t {\ve x}_t = {\ve 0}.
\end{equation*}

We will show that ${\ve b} = \mu_1 {\ve x}_1 + \cdots + \mu_t {\ve x}_t$, $\mu_i \in \ZZ_{\ge 0}$ with $\mu_i = 0$ for at least one $i$.  Indeed, this means there exists a proper subset $Y$ of $X$ such that ${\ve b} \in \icone(Y)$.

Suppose that ${\ve b} = \lambda_1 {\ve x}_1 + \cdots + \lambda_t {\ve x}_t$, $\lambda_i \in \ZZ_{> 0}$.  Writing $\lambda = \min_{i: y_i\neq 0} \lambda_i$ and replacing, if necessary, the vector ${\ve y}$ by $-{\ve y}$, we have
\begin{equation*}
{\ve b} = \lambda_1 {\ve x}_1 + \cdots + \lambda_t {\ve x}_t - \lambda(y_1 {\ve x}_1 + \cdots + y_t {\ve x}_t)
= \sum_{i=1}^t (\lambda_i-\lambda y_i){\ve x}_i\,,
\end{equation*}
where all coefficients $\mu_i:=\lambda_i-\lambda y_i$ are nonnegative and at least one of them is zero, as desired.
\end{proof}

The second ingredient for the proof of Theorem \ref{t:Bound_via_determinant} is the following technical result.

\begin{lemma} \label{projection_lemma} Let $X \subset \ZZ^d$ be a finite set of nonzero integer vectors and $Y$ be a subset of $X$. Then $H(Y)\le H(X)$.
\end{lemma}

\begin{proof}[Proof of Lemma~\ref{projection_lemma}]
For $m \in \ZZ_{\ge 1}$ we denote by $[m]$ the set $\{1,2,\ldots,m\}$. Given a matrix $V\in\Z^{ m\times n}$ and a set $I\subset[m]$ we denote by $V_I$ the matrix with rows $i \in I$ of $V$.

Assume without loss of generality that $Y$ consists of the first $s$ elements of $X$.
Let $V\in\Z^{t \times r}$ be a matrix  with $r=r(X)$ columns that form a basis of the lattice $\Lambda=\Lambda(X)=V \Z^r$.
Let $\Gamma$ denote the projection of $\Lambda$ onto the first $s$ coordinates, i.e., $\Gamma=V_{[s]}\Z^r$.
We have $\spn_{\R}(W(Y))=\spn_{\R}(W(X)_{[s]})=\spn_{\R}(V_{[s]})$ and, consequently, $\Gamma\subset\Lambda(Y)$. Thus, $H(Y)\le \det(\Gamma)$.
Since $H(X)=\det(\Lambda)$, it is sufficient to show that
\begin{equation}\label{sufficient_inequ}
\det(\Gamma)\le \det(\Lambda)\,.
\end{equation}

We have that
\begin{equation*}\det(\Lambda)=\sqrt{\det(V^T V)}=\sqrt{\sum_{I\subset[t]\text{ s.t. }|I|=r}\det(V_I)^2}
\end{equation*}
(see e.g. \cite{MacLaneBirkhoff}, Chapter 16, Theorem 18.)
 Note that by the choice of $V$, we have $g=1$ in \eqref{SmithExpr}.

We denote $B=V_{[s]}$ and $k=\rank(B)$. Further, let $C=V_{[k]}$ and $D=V_{[s+r-k]\setminus[s]}$.
By permutation of the rows of $V$ we may assume without loss of generality that $\rank(B)=\rank(C)=k$ and $\rank((C^T\,D^T))=r$

Now let $U\in\Z^{r \times r}$ be unimodular such that the matrix $(C^T\,D^T)^TU$ is in the Hermite normal form (see \cite[Section 4.1]{schrijver}).
In particular $((C^T\,D^T)^TU)_{ij}=0$ for $i<j$.

Let $V'=VU$, $B'=BU$, $C'=CU$ and $D'=DU$.
Observe that $V'\Z^r=VU\Z^r =\Lambda$ and $B'\Z^r=BU\Z^r=\Gamma$ and
\begin{equation*}\sum_{I\subset[t]\text{ s.t. }|I|=r}\det(V_I)^2=\sum_{I\subset[t]\text{ s.t. }|I|=r}\det(V_IU)^2.
\end{equation*}
In particular $V'_{ij}=0$ for $i=1,\ldots,s$ and $j=1,\ldots,r-k$, other wise $\rank(B)>k$.

Let $\bar B$ denote the $k$ non-zero columns of $B'$.
It holds that $\bar B\Z^k=\Gamma$.
Let $J=\{s+1,\ldots,s+r-k\}$ and let $Id_n$ denote the $n\times n$ identity matrix.
Then
\begin{equation*}\begin{array}{rcl}\
(\det(\Gamma))^2 & = & \sum_{I\subset[s]\text{ s.t. }|I|=k}\det(\bar B_I)^2\\[0.5cm]
& = & \sum_{I\subset[s]\text{ s.t. }|I|=k}\det\left(\left(\begin{matrix}0 & \bar B_I \\ Id_{r-k} & 0 \end{matrix}\right)\right)^2 \\[0.5cm]
& \le & \sum_{I\subset[s]\text{ s.t. }|I|=k}\det\left(\left(\begin{matrix} B'_I \\ D' \end{matrix}\right)\right)^2 \\[0.5cm]
& = & \sum_{I\subset[s]\text{ s.t. }|I|=k}\det\left(V'_{I \cup J}\right)^2 \\[0.5cm]
& \le & \sum_{I\subset[t]\text{ s.t. }|I|=r}\det(V'_I)^2 =(\det(\Lambda))^2.\\
\end{array}
\end{equation*}
Hence (\ref{sufficient_inequ}) holds and the lemma is proved.

\end{proof}

Now we are ready to prove Theorem \ref{t:Bound_via_determinant}. 

\begin{proof}[Proof of Theorem~\ref{t:Bound_via_determinant}]

It is sufficient to show that
\begin{equation}\label{relax}
{\mathsf M}_0(X) \le r(X) + \log(H(X))\,.
\end{equation}
Suppose, to derive a contradiction, that ${\mathsf M}_0(X) > r(X) + \log(H(X))$. Then there exists ${\ve b}\in \icone(X)$ such that
\begin{equation*}\label{nonsparse}
{\mathsf m}_0({\ve b})>r(X) + \log(H(X)).
\end{equation*}
Let $Y$ be a subset of $X$ such that ${\ve b}\in\icone(Y)$ and
${\mathsf m}_0({\ve b})=|Y|$.
Observe that $|Y| > r(X)+ \log(H(X))$ implies $|Y| > r(Y)+ \log(H(Y))$.
Indeed, we clearly have $r(Y) \le r(X)$ and the inequality $\log(H(Y))\le \log(H(X))$ follows from Lemma \ref{projection_lemma}.
Therefore, by Lemma \ref{t:Lemma_ES_refinement},  ${\mathsf m}_0({\ve b})<|Y|$. The obtained contradiction implies (\ref{relax}) and completes the proof of part (i).

Similarly, to prove part (ii) it suffices to show that for any subset $Y$ of $X$ the inequality $|Y| > 2d\log(2\sqrt{d}||X||_{\infty})$ implies
%
\begin{equation*}|Y| > r(Y)+ \log(H(Y)).\,
\end{equation*}
%
First, observe that
\begin{equation}\label{eq:technical_bounds}
r(Y) \le d \;\; \mbox{and} \;\; H(Y) \le (\sqrt{|Y|}||X||_{\infty})^d\,.
\end{equation}
Suppose that $|Y| > 2d\log(2\sqrt{d}||X||_{\infty})$, that is,
\begin{equation*}||X||_{\infty}<\frac{1}{\sqrt{2d}}2^{\frac{|Y|-d}{2d}}\,.
\end{equation*}
By (\ref{eq:technical_bounds}), we see that
\begin{equation*}\begin{array}{rcl}
r(Y) + \log(H(Y))
&\le& d + d\log(\sqrt{|Y|}||X||_{\infty}) \\
&<& d + d\log\left(\sqrt{|Y|}\frac{1}{\sqrt{2d}}2^{\frac{|Y|-d}{2d}}\right) \\
&=& \frac{|Y|}{2}+\frac{d}{2}\log(\frac{|Y|}{d}) < |Y|\,,
\end{array}
\end{equation*}
which completes the proof of part (ii).
\end{proof}

\subsection{Knapsack case ($d = 1$) and sum-distinct sets}%
\label{Knapsack_subsection}

In this section we
prove Theorem \ref{t:Lemma_Positive_Knapsack_refinement} and propose a conjecture on the bounds for ${\mathsf M}_0(X)$ in terms of $\|X\|_{\infty}$ in the knapsack case.

The proof of Theorem \ref{t:Lemma_Positive_Knapsack_refinement} will easily follow from the following lemma.

\begin{lemma}\label{t:Lemma_knapsack_refinement}
Let $X$ be a finite set of positive integers and let ${b} \in \icone(X)$. If
\begin{equation}\label{eq:Bound_via_minor}
|X| > 1 + \log(\|X\|_{\infty})\,
\end{equation}
then there exists a proper subset $Y\subset X$ such that ${b} \in \icone(Y)$.
\end{lemma}

\begin{proof}[Proof of Lemma~\ref{t:Lemma_knapsack_refinement} ]
Let  $X = \{x_1,\ldots, x_t\} \subset \ZZ^{t}_{>0}$  with $t \ge 2$.
%
%
Consider the {\em knapsack polytope}
\begin{equation*}
Q_X(b)=\{{\ve y}\in \R^t_{\ge 0}: W(X)^T {\ve y}=b\}\,.
\end{equation*}
The polytope $Q_X( b)$ is a $(t-1)$-dimensional simplex in $\R^t$ with vertices
\begin{equation*}
(b/x_1,0, \ldots, 0)^T, (0, b/x_2, \ldots, 0)^T, \ldots, (0, \ldots, 0, b/x_t)^T\,.
\end{equation*}

We will show that there exists an integer point point ${\ve \mu}=(\mu_1 , \ldots , \mu_t)^T\in Q_X( b)$ with $\mu_i = 0$ for at least one $i$.  Indeed, this means that there exists a proper subset $Y$ of $X$ such that $b \in \icone(Y)$.

We will work with the $(t-1)$-dimensional simplex
$
S_X(b)= \pi_t(Q_X(b))\subset \R^{t-1}\,,
$
where $\pi_t(\cdot): \R^t\rightarrow \R^{t-1}$ is the projection that forgets the last coordinate.
Observe that
\begin{equation*}
S_X( b) = \left\{ {\ve y} \in \R_{\geq 0}^{t-1} : { y}_1\,x_1+\cdots +{ y}_{t-1}\,x_{t-1}\leq b \right\}\,
\end{equation*}
and, as all $x_i>0$, the projection map $\pi_t(\cdot)$ establishes a bijection between $Q_X( b)$ and $S_X(b)$.

Let ${\ve \lambda}$ be any integer point in $Q_X( b)$. Suppose that ${\ve \lambda}\in \Z^t_{>0}$.
Clearly, ${\ve \beta}=\pi_t({\ve \lambda})\in S_X( b)$.
Consider the lattice
\begin{equation*}
\Gamma=\{{\ve y}\in \Z^{t-1}: y_1 x_1+\cdots+ y_{t-1}x_{t-1}\equiv 0 (\modulo x_t)\}\,.
\end{equation*}
The lattice $\Gamma$ has determinant $\det(\Gamma)= x_t/g$, where $g=\gcd(x_1,\ldots, x_t)$ (see e.g. Corollary 3.2.20 in \cite{Dwork}).


By Minkowski's first fundamental theorem (see e.g. \cite{GrLek}), applied to $\Gamma$ and the cube $[-1,1]^{t-1}$, there exists a nonzero ${\ve u}\in \Gamma$ such that
\begin{equation}\label{Minkowski_bound_knapsack}
\|{\ve u}\|_{\infty}\le (\det(\Gamma))^{1/(t-1)}
\le x_t^{1/(t-1)}\,.
\end{equation}
By (\ref{eq:Bound_via_minor}), we have $x_t< 2^{t-1}$. Hence, together with (\ref{Minkowski_bound_knapsack}), the point ${\ve u} = (u_1, \ldots, u_{t-1})^T$ has entries $u_i \in \{-1,0,1\}$.
For some integer $z$, the point ${\ve p}=(u_1, \ldots, u_{n-1}, z)^T$ satisfies
\begin{equation}\label{inaffinelattice}
W(X)^T{\ve p}=0\,.
\end{equation}

Suppose first that all nonzero entries of ${\ve u}$ are of the same sign.
In this case we may assume without loss of generality that all nonzero
entries of ${\ve u}$ are positive. Then for some integer $k$ the point
${\ve \beta} - k{\ve u}\in S_X( b)$ will have less than $t-1$ nonzero entries.
Therefore, the point
${\ve \mu}={\ve \lambda} - k{\ve p}$ will have less than $t$ nonzero entries.
By construction, $\pi_t({\ve \mu})\in S_X(b)$. On the other hand, by
(\ref{inaffinelattice}), $W(X)^T{\ve \mu}=b$. This implies ${\ve \mu}\in Q_X( b)$.
Hence this case is settled.

Suppose now that there are entries $u_i$ and $u_j$
with $u_i u_j<0$.
Then for some positive integers $k, l$ the points  ${\ve \beta} - k{\ve u}, {\ve \beta} + l{\ve u} \in  \Z^{t-1}_{\ge 0}$
will have each less than $t-1$ nonzero entries.
Clearly, at least one of these points is in the simplex $S_X( b)$. If  ${\ve \beta} - k{\ve u}\in S_X( b)$
then set ${\ve \mu}={\ve \lambda} - k{\ve p}$. Otherwise, set ${\ve \mu}={\ve \lambda} + l{\ve p}$.
Hence, as in the  previous case, the point
${\ve \mu}\in Q_X( b)$ will have less than $t$ nonzero entries. The lemma is proved.

\end{proof}

\begin{proof}[Proof of Theorem~\ref{t:Lemma_Positive_Knapsack_refinement} ]

It is sufficient to show that
\begin{equation}\label{relax_knapsack}
{\mathsf M}_0(X) \le 1 + \log(\|X\|_{\infty})\,.
\end{equation}
Suppose, to derive a contradition, that ${\mathsf M}_0(X) > 1 + \log(\|X\|_{\infty})$. Then there exists ${b}\in \icone(X)$ such that
\begin{equation}\label{nonsparse_knapsack}
{\mathsf m}_0({b})>1 + \log(\|X\|_{\infty}).
\end{equation}
Let $Y$ be a subset of $X$ such that ${ b}\in\icone(Y)$ and
$
{\mathsf m}_0({ b})=|Y|.
$
By (\ref{nonsparse_knapsack}) we have $|Y|>1 + \log(\|Y\|_{\infty})$. Therefore, by  Lemma \ref{t:Lemma_knapsack_refinement},  ${\mathsf m}_0({ b})<|Y|$. The obtained contradiction implies
(\ref{relax_knapsack}).
\end{proof}


We now  discuss the bounds for ${\mathsf M}_0(X)$ in terms of $\|X\|_{\infty}$ for the general knapsack problem.
A minor improvement on part (ii) of Theorem \ref{t:Bound_via_determinant} for $d=1$ can be obtained by using a version of Siegel's Lemma that works with the maximum norm.
Let $A\in \ZZ^{1\times n}$, $n\ge 2$, be a  matrix with nonzero entries.  Siegel's Lemma with respect to the maximum norm asks for the smallest possible constant $c_n > 0$ such that the equation $A{\ve y} = 0$ has a solution ${\ve y} \in \ZZ^n$ with
%
\begin{equation*}0 < ||{\ve y}||_\infty^{n-1} \le c_n||A||_\infty\,.
\end{equation*}
%
The only known exact values of $c_n$ are $c_2 = 1$, $c_3 = 4/3$ and $c_4=27/19$
(see \cite{NewSL}).  The upper bound $c_n\le\sqrt{n}$ immediately follows from Theorem~\ref{t:SiegelsLemma}. It was shown in \cite{SLsumdistinct} that $c_n\le \sigma_n^{-1}$,
where $\sigma_n$ denotes the \emph{sinc} integral
\begin{equation*}\sigma_n = \frac{2}{\pi}\int_0^\infty \left(\frac{\sin x}{x}\right)^n dx\,.
\end{equation*}
Since $\sigma_n^{-1}\sim\sqrt{\frac{\pi n}{6}}$ as $n\rightarrow\infty$, the latter bound asymptotically improves the estimate $c_n\le\sqrt{n}$.
The sequences of numerators and denominators of $\sigma_n/2$ can be found in \cite{Sloane} as sequences A049330 and A049331, respectively.
Using the bound  $c_n\le \sigma_n^{-1}$, we obtain the following version of Lemma \ref{t:Lemma_ES_refinement} in terms of $\|X\|_{\infty}$.

%
\begin{lemma}\label{t:Lemma_Knapsack_refinement}
Let $X$ be a finite set of integers and let $b\in \icone(X)$. If
\begin{equation}\label{eq:Knapsack_Bound_via_determinant}
|X| > 1 -\log(\sigma_{|X|}) + \log(||X||_{\infty})\,.
\end{equation}
then there exists a proper subset $Y \subset X$ such that $b \in \icone(Y)$.
\end{lemma}

\begin{proof}[Proof of Lemma~\ref{t:Lemma_Knapsack_refinement}]
Writing $t = |X|$, the inequality (\ref{eq:Knapsack_Bound_via_determinant}) implies that
\begin{equation*}||X||_{\infty} < \sigma_t 2^{t-1}.
\end{equation*}
By {\cite[Theorem~1]{SLsumdistinct}}, there exists a nonzero ${\ve y} \in \ZZ^t$ such that $y_1 x_1 + \cdots + y_t x_t = 0$ and
%
\begin{equation}\label{less_than_2}||{\ve y}||_\infty \le \left(\frac{||X||_{\infty}}{\sigma_t}\right)^{\frac{1}{t-1}} <  2\,.
\end{equation}
%
As in the proof of Lemma~\ref{t:Lemma_ES_refinement} we now observe that ${\ve y} = (y_1, \ldots, y_t)^T$ has entries $y_i \in \{-1,0,1\}$, yielding the proper subset $Y$ of $X$.
\end{proof}
Lemma \ref{t:Lemma_Knapsack_refinement}
can be used to obtain a minor asymptotic improvement on part (ii) of Theorem \ref{t:Bound_via_determinant} for $d=1$.


It was observed by Schinzel (personal communication) that Siegel's Lemma is closely related to the following well known problem from additive number theory.  A finite set $X = \{x_1, \ldots, x_t\} \subset \ZZ$ of integers is called \emph{sum--distinct} if any two of its $2^t$ subsums differ by at least $1$.  We shall assume w.l.o.g.\ that $0 < x_1 < x_2 < \cdots < x_t$.  In 1955, Erd\H{o}s and  Moser (\cite{EM}, Problem 6) asked for an estimate on the least possible $x_t$ of such a set, and Erd\H{o}s conjectured that $x_t > c_0 2^t$ for some absolute constant $c_0 > 0$.
It follows from {\cite[Theorem~1]{SLsumdistinct}} that a sum-distinct set $X=\{x_1, \ldots, x_t\}$ satisfies
$ ||X||_{\infty} > \sigma_t 2^{t-1}$
(cf. (\ref{less_than_2}).)
An unpublished result by Elkies and Gleason asymptotically improves this bound by a factor of $2/\sqrt{3}$.

Observe that the set $X$ is not sum-distinct if and only if  there exist $y_i\in \{-1,0,1\}$, not all zero, such that $y_1 x_1 + \cdots + y_t x_t = 0$.
Hence, an affirmative answer to Erd\H{o}s-Moser conjecture would imply that (\ref{eq:Knapsack_Bound_via_determinant}) can be replaced by the inequality
\begin{equation*}\label{IfEM}
|X|\ge  -\log(c_0) + \log(||X||_{\infty})\,.
\end{equation*}
Lemma \ref{t:Lemma_knapsack_refinement} allows us to make such a replacement for sets of positive integers.
Based on this, we conjecture that for $d=1$ there exists a positive integer $c$ such that
\begin{equation*}
{\mathsf M}_0(X)\le c+\lfloor\log(||X||_{\infty})) \rfloor\,.
\end{equation*}

As a final remark for Section \ref{s:caratheodory} we wish to discuss how tight are the bounds we have presented.
The bound in part (i) of Theorem \ref{t:Bound_via_determinant}
is in fact attained for certain sets of generators in the cases $d=2$ and $d=3$. To see this,
following \cite{EisenbrandShmonin2006}, we consider the sets $X\subset \Z^d$ defined for $n\ge 2$ as
\begin{equation*}
X=\{{\ve x}_{ij}: {\ve x}_{ij}=2^i {\ve e}_j+{\ve e}_d\,, i=0,\ldots, n-1\,, j=1, \ldots, d-1\}\,,
\end{equation*}
where ${\ve e}_j$ is the $j$th standard basis vector.

For ${\ve b}=(2^{n}-1)\sum_{j=1}^{d-1}{\ve e}_j+ n(d-1){\ve e}_d$ we clearly have ${\mathsf m}_0({\ve b})=n(d-1)$ and, consequently, ${\mathsf M}_0(X)=n(d-1)$.
This value coincides with the upper bound (i) of Theorem \ref{t:Bound_via_determinant} for $d=2$ with $n=2,3$ and for $d=3$ with $n=2$.


\section{Periodicity of the function $\mathsf m_0$}
\label{s:multiquasi}

In polyhedral combinatorics and combinatorial optimization it is well-known that there are interesting properties of dilation of polyhedra. For instance, the reader may be familiar with functions that behave well under dilation, such as the volume and the number of lattice points; their behaviour is captured by Ehrhart's Theorem (see \cite{BarviPom} for an introduction).  In this section we explain  the asymptotic behavior of the norm-based function $\mathsf m_0(\vv b)$, both in the most general setting (Theorem~\ref{t:eventualm0affine}) and for knapsack problems (Theorem~\ref{t:eventualm0knap}).  We prove in particular that $\mathsf m_0:\icone(X) \to \ZZ_{\ge 0}$ is eventually periodic. Let us consider a motivating example, a simple knapsack problem.

\begin{example}\label{e:eventualm0knap}
Let $X = \{4,6,15\} \subset \ZZ$, and let $S = \icone(X)$.  Depicted in Figure~\ref{f:m0_polytope} are the polytopes $Q_X(25)$ and $Q_X(85)$.  There is a single integer solution in $\mathsf P_X(85)$ that lies on a coordinate plane, namely $(10,0,3)^T$, so $\mathsf m_0(85) = 2$.  On the other hand, the only integer solution in $\mathsf P_X(25)$ is $(1,1,1)^T$, so $\mathsf m_0(25) = 3$.

Notice that $85 - 25 = 60 = \lcm(X)$, so any integer solution in $\mathsf P_X(25)$ produces a solution in $\mathsf P_X(85)$ by sufficiently increasing any single component.
Geometrically, the drop in 0-norm value occurs because the point $(10,0,3)^T \in \mathsf P_X(85)$ cannot be obtained in this way from a solution in $\mathsf P_X(25)$.
\end{example}

\begin{figure}[ht]
\begin{center}
\includegraphics[width=3in]{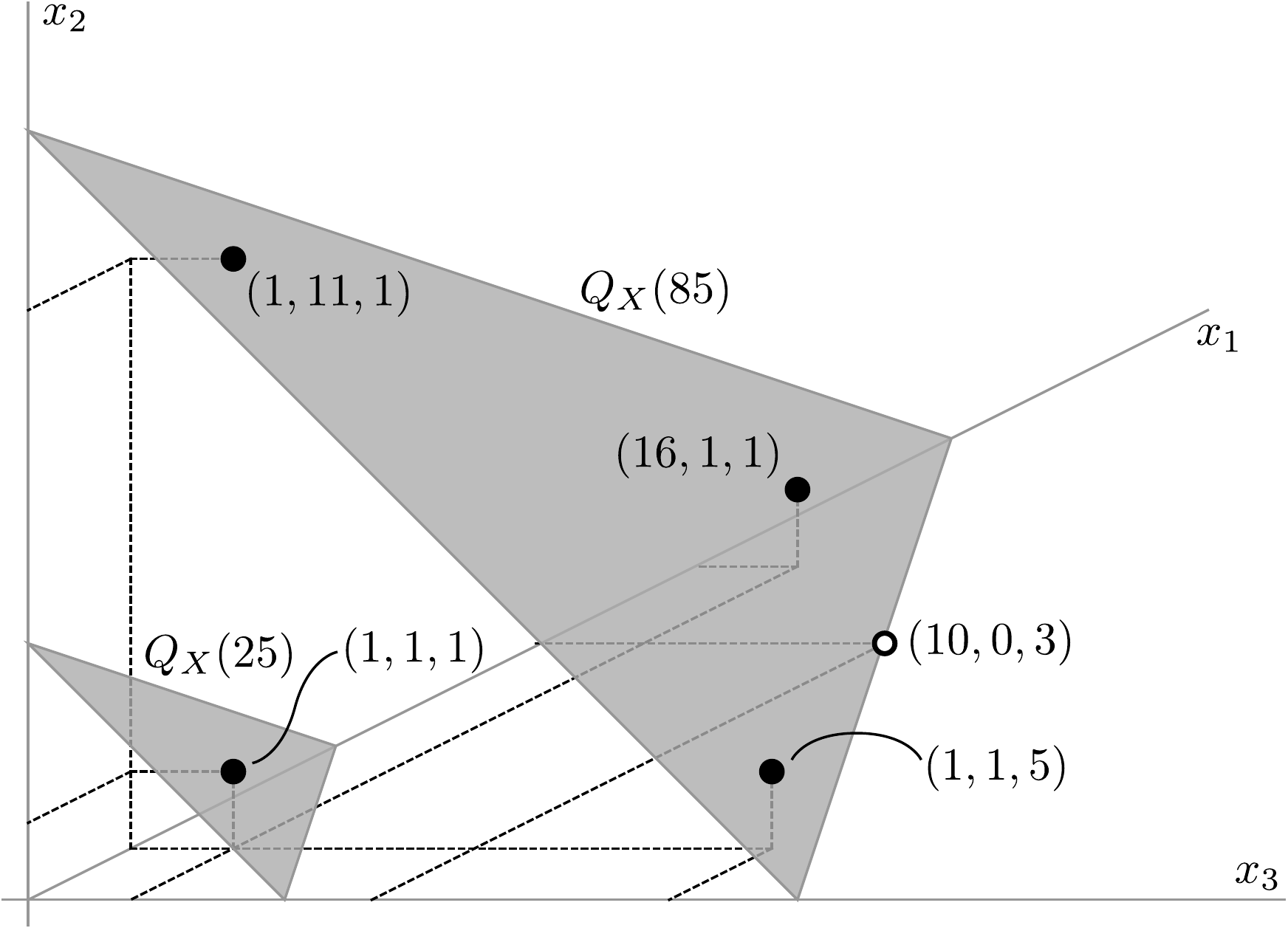}
\end{center}
\caption{Polytopes $Q_X(25)$ and $Q_X(85)$, with $X = \{4,6,15\}$ as in Example~\ref{e:eventualm0knap}.}
\label{f:m0_polytope}
\end{figure}


Recall that a function $f:\ZZ_{\ge 0} \to \QQ$ is a \emph{quasipolynomial of degree $k$} if there exist periodic functions $a_0, \ldots, a_k:\ZZ_{\ge 0} \to \QQ$ such that
$$f(b) = a_k(b)b^k + \cdots + a_1(b)b + a_0(b)$$
and $a_k$ is not identically zero.  The \emph{period of $f$} is the minimal positive integer $\pi$ such that $a_i(b + \pi) = a_i(b)$ for all $i \le k$ and $b \in \ZZ_{\ge 0}$.  The statement of Theorem~\ref{t:eventualm0affine} requires an appropriate multivariate analog.

\begin{defn}\label{d:eventualquasipolynomial}
Fix $f:\ZZ_{\ge 0}^d \to \RR$, linearly independent $\vv x_1, \ldots, \vv x_t \in \ZZ_{\ge 0}^d$, and $\vv b \in \ZZ_{\ge 0}^d$.
\begin{enumerate}[(a)]
\item
The \emph{cone generated by $\vv x_1, \ldots, \vv x_t$ translated by $\vv b$} is the set
\begin{center}
$C = C(\vv b;\vv x_1, \ldots, \vv x_t) = \vv b + \icone(\{\vv x_1, \ldots, \vv x_t\}) \subset \ZZ_{\ge 0}^d$.
\end{center}

\item
The function $f$ is a \emph{simple quasipolynomial supported on the cone $C$} if (i) $f$ vanishes outside of $C$ and (ii) $f$ coincides with a polynomial when restricted to $C$.  The \emph{degree of $f$} is the degree of the restriction of $f$ to $C$.


\item
The function $f$ is \emph{eventually quasipolynomial} if it is a finite sum of simple quasipolynomials.  The \emph{degree} of $f$ is the minimal integer $k$ such that $f$ can be written as a sum of simple quasipolynomials of degree at most $k$.

\end{enumerate}
\end{defn}

We are now ready to state the main result for this section (Theorem~\ref{t:eventualm0affine}).

\begin{thm}\label{t:eventualm0affine}
For any finite set of non-negative integer vectors $X \subset \ZZ^d$, $\mathsf m_0:\icone(X) \to \ZZ_{\ge 0}$ is eventually quasiconstant, i.e.\ a quasipolynomial of degree $0$.
\end{thm}

Focusing our attention on the case of knapsack problems (that is, when $X \subset \ZZ_{> 0}$), Theorem~\ref{t:eventualm0affine} implies that the function $\mathsf m_0(b)$ is eventually periodic.  Our next main result, Theorem~\ref{t:eventualm0knap}, gives a more direct proof of this fact; in doing so, a bound on the starting point of this periodic behavior and a precise value for the minimal period are achieved.

In what follows, for $X' \subset \ZZ_{> 0}$, let
$$F(X') = \max(\icone(\gcd(X')) \setminus \icone(X')),$$
which coincides with the {\em Frobenius number} (see e.g. \cite{JRA}) when $\gcd(X') = 1$.

\begin{thm}\label{t:eventualm0knap}
Fix a set $X$ of positive integers, let $L = \lcm(X)$, and let
$$N_0 = \max\{F(X') : X' \subset X\}.$$
Each $b > N_0$ satisfies
$$\mathsf m_0(b + L) = \mathsf m_0(b).$$
Moreover, $\lcm(X)$ is the minimal value of $L$ for which the above statement holds.
\end{thm}

\begin{example}
Resuming notation from Example \ref{e:eventualm0knap}, Figure~\ref{f:m0_plot} depicts values of the function $\mathsf m_0:\icone(X) \to \ZZ_{\ge 0}$. The periodic behavior ensured by Theorems~\ref{t:eventualm0affine} and~\ref{t:eventualm0knap} is evident, as we can see that $\mathsf m_0(b + 60) = \mathsf m_0(b)$ for each $b \ge 42$.

Algebraically, this occur because $L = 60$ lies in every subsemigroup of $\icone(X)$ generated by a proper subset of $X$, so adding $L$ to any element of $\icone(X)$ preserves membership in all such subsemigroups.  We must require $b$ be sufficiently large to ensure $b + 60$ is not contained in any additional subsemigroups not already containing $b$ (Figure~\ref{f:m0_polytope} depicts this phenomenon for $b = 25$).
\end{example}

\begin{figure}[h]
\begin{center}
\includegraphics[width=5.6in]{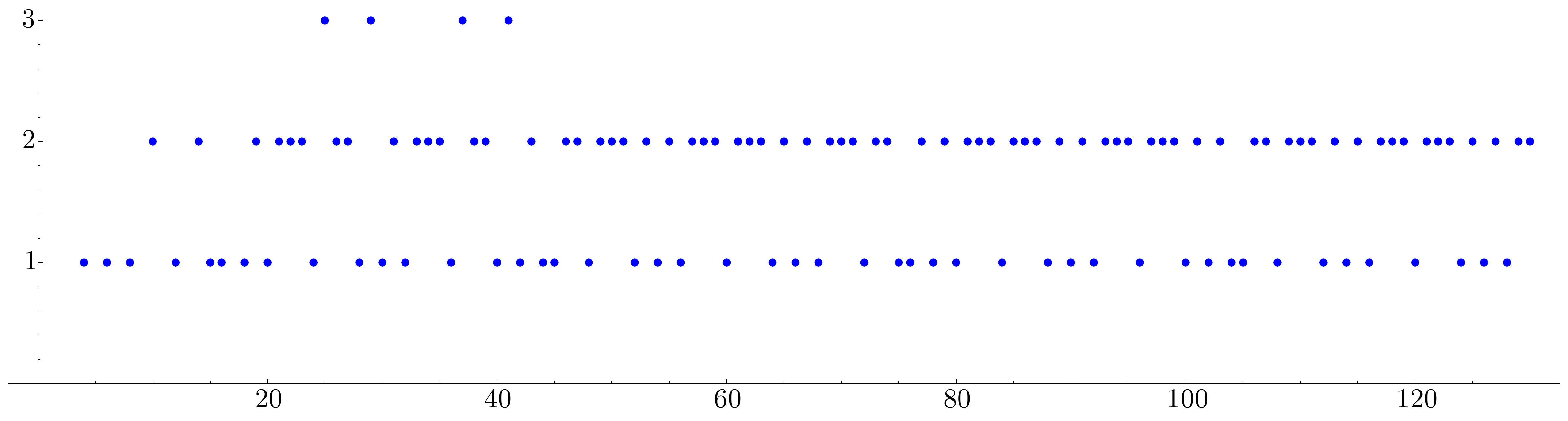}
\end{center}
\caption{Plot of $\mathsf m_0:\icone(X) \to \ZZ_{\ge 0}$ for $X = \{4,6,15\}$ in Example~\ref{e:eventualm0knap}.}
\label{f:m0_plot}
\end{figure}

Our study of the asymptotic behavior of $\mathsf m_0(\vv b)$ was inspired by \cite[Theorem~5.2]{factorhilbert}, a result that appeared in \cite{elastsets} for knapsack problems, stating that the value of the maximal $\ell_1$-norm is eventually quasilinear (Definition~\ref{d:eventualquasipolynomial}). In future work we will discuss the situation for other norms. The proof of Theorem \ref{t:eventualm0affine}, presented later in this section, uses generating functions and Hilbert's Theorem (Theorems~\ref{t:hilbert} and~\ref{t:hilbert2}).

\subsection{Proof of Theorems \ref{t:eventualm0affine} and \ref{t:eventualm0knap} through Commutative Algebra}

The following theorem, a consequence of the Hilbert Basis Theorem, characterizes the eventual behavior of Hilbert functions of certain $\ZZ_{\ge 0}$-graded modules.  For more background on Hilbert functions and graded modules, see \cite{cca,caca}.

\begin{defn}\label{d:hilbert}
Fix a field $\kk$ (if the reader prefer, he/she can assume $\kk=F_2$, but what follows works in full generality), let $R = \kk[z_1, \ldots, z_t]$, and fix an $R$-module $M$.  A \emph{$\ZZ_{\ge 0}^d$-grading of $R$} is a function $\deg:\ZZ_{\ge 0}^t
\to \ZZ_{\ge 0}^d$ satisfying
$$\deg(\vv y + \vv y') = \deg(\vv y) + \deg(\vv y')$$
for all $\vv y, \vv y' \in \ZZ_{\ge 0}^t$.  Here, $\deg(z_1^{y_1} \cdots z_t^{y_t}) = \deg(\vv y)$ represents the degree of the monomial $z_1^{y_1} \cdots z_t^{y_t}$ for $\vv y \in \ZZ_{\ge 0}^t$.  Let $R_\vv b$ denote the $\kk$-vector subspace of $R$ spanned by those $\vv y \in \ZZ_{\ge 0}^t$ satisfying $\deg \vv y = \vv b$.
An \emph{$\ZZ_{\ge 0}^d$-grading of $M$} is an expression
$$M \iso \bigoplus_{\vv b \in \ZZ_{\ge 0}^d} M_\vv b$$
of $M$ as a direct sum of finite dimensional $\kk$-subspaces of $M$ with $R_\vv b M_{\vv b'} \subset M_{\vv b + \vv b'}$ for all $\vv b, \vv b' \in \ZZ_{\ge 0}^d$.  The \emph{Hilbert function of $M$} is the function $\mathcal H(M;-):\ZZ_{\ge 0}^d \to \ZZ_{\ge 0}$ given by
$$\mathcal H(M;\vv b) = \dim_\kk M_\vv b$$
for each $\vv b \in \ZZ_{\ge 0}^d$.
\end{defn}

\begin{thm}[{\cite[Theorem~I.2.3]{caca}}]\label{t:hilbert}
Fix a $\ZZ_{\ge 0}$-graded polynomial ring $R$, and a finitely generated $\ZZ_{\ge 0}$-graded $R$-module $M$ of Krull dimension $d$.  The Hilbert function of $M$ eventually equals a quasipolynomial of degree $d-1$ (called the \emph{Hilbert quasipolynomial of $M$}).  More specifically, there exist periodic functions $a_0, \ldots, a_{d-1}:\ZZ_{\ge 0} \to \QQ$ such that $a_{d-1} \not\equiv 0$ and
$$\mathcal H(M;b) = a_{d-1}(b)b^{d-1} + \cdots + a_1(b)b + a_0(b)$$
for sufficiently large $b$.  Additionally, if $y_1, \ldots, y_d \in R$ is a homogeneous system of parameters for $M$, then the period of each $a_i$ divides $\lcm(\deg(y_1), \ldots, \deg(y_d))$.
\end{thm}



The proof of Theorem~\ref{t:eventualm0affine} uses a generalization of Theorem~\ref{t:hilbert} to multigradings.

\begin{thm}[{\cite[Theorem~2.10]{factorhilbert}}]\label{t:hilbert2}
Fix a $\ZZ_{\ge 0}^d$-graded polynomial ring $R$, and a finitely generated $\ZZ_{\ge 0}^d$-graded $R$-module $M$.  The Hilbert function of $M$ is eventually quasipolynomial.
\end{thm}

We are now ready to prove Theorems~\ref{t:eventualm0affine} and~\ref{t:eventualm0knap}.

\begin{proof}[Proof of Theorem~\ref{t:eventualm0affine}]
We will use Theorem~\ref{t:hilbert2}.  Multigrade the polynomial ring $\kk[z_1, \ldots, z_t]$ by $\deg(z_i) = \vv x_i$ for $i \le t$.  The ideals
$$I_j = \left\<\textstyle\prod_{i \in T} z_i : T \subset \{1, \ldots, t\}, |T| = j\right\> \subset \kk[z_1, \ldots, z_t]$$
for $j \le t$ and $I_{t+1} = 0$ form a descending chain
$$I_1 \supset I_2 \supset \cdots \supset I_t \supset I_{t+1} = 0$$
with the property that for $j \le t$, $\mathcal H(I_j/I_{j+1};\vv b) > 0$ if and only if $\vv b \in \icone(X)$ has a solution $\vv y \in \mathsf P_X(\vv b)$ with $\|\vv y\|_0 = j$.  Indeed, each $I_j$ is determined by the monomials it contains, and a solution $\vv y \in \mathsf P_X(\vv b)$ satisfies $\|\vv y\|_0 = j$ if and only if $z_1^{y_1} \cdots z_t^{y_t} \in I_j \setminus I_{j+1}$.

Applying Theorem~\ref{t:hilbert2} to the quotient $I_j/I_{j+1}$ proves that the characteristic function
$$\chi_j(\vv b) = \left\{\begin{array}{ll}
1 & \text{ if } \|\vv y\|_0 = j \text{ for some } \vv y \in \mathsf P_X(\vv b) \\
0 & \text{ otherwise }
\end{array}\right.$$
is eventually quasiconstant for each $j \le t$, and thus
$$
\mathsf m_0(\vv b) = \max\{(t-j)\chi_j(\vv b) : j \le t\}$$
is eventually quasiconstant as well.
\end{proof}

\begin{proof}[Proof of Theorem~\ref{t:eventualm0knap}]
First, fix $\vv y \in \mathsf P_X(b)$.  Since $b > 0$, we have $y_j > 0$ for some $j \le t$.  Since $x_j \mid L$, the solution $\vv y + (L/x_j)\ee_j \in \mathsf P_X(b + L)$ has support $\supp(\vv y)$.  This proves $\mathsf m_0(b + L) \le \mathsf m_0(b)$.  Next, fix $\vv y \in \mathsf P_X(b + L)$ with $|\supp(\vv y)|$ minimal, and let $S' = \icone(\supp(\vv y))$.  Since $b > F(S')$, we have $b \in S'$ and thus $\mathsf m_0(b) \le |\supp(\vv y)|$.  We conclude $\mathsf m_0(b + L) = \mathsf m_0(b)$.

For the final claim, consider the characteristic function $\chi_T$ of the set $T = \bigcup_{i \le t} \icone(x_i)$.
Notice that $\chi_T(b)$ is nonzero precisely when $\mathsf m_0(b) = 1$.  The result follows from the fact that $\chi_T$ has minimal period $\lcm(X)$.
\end{proof}

\begin{remark}\label{r:eventualm0knap}
Notice that simply proving ``$\supseteq$'' in Theorem~\ref{t:eventualm0knap} is sufficient to prove eventual periodicity of $\mathsf m_0$ if a precise lower bound on the start of periodicity is not desired.  Indeed, it follows that for each $j < L$, the sequence $\{\mathsf m_0(kL + j)\}_{k \ge 1}$ is (eventually) a non-increasing sequence of positive integers, and thus eventually constant.
\end{remark}

We conclude this section by considering the following problem: ``Give a bound on $\mathsf m_0(b)$ that holds for all but finitely many $b \in \icone(X)$.''  Proposition~\ref{p:eventualboundm0knap} gives an optimal answer in the knapsack setting.

\begin{prop}\label{p:eventualboundm0knap}
If $m$ is the minimal size of a relatively prime subset of $X \subset \ZZ_{\ge 0}$, then $\mathsf m_0(b) \le m$ for all but finitely many $b \in \icone(X)$.  More specifically, let $X_1, \ldots, X_r$ denote the relatively prime subsets of $X$ with cardinality $m$, and let
$$N_0 = \max\left(\textstyle\bigcap_{i = 1}^r \ZZ_{\ge 0} \setminus \icone(X_i)\right).$$
Then $\mathsf m_0(b) \le m$ for all $b > N_0$, and $\mathsf m_0(b) = m$ for infinitely many $b > N_0$ (that is to say, the eventual bound $m$ is sharp).
\end{prop}

\begin{proof}
Every $b > N_0$ lies in $\<X_i\>$ for some $i \le r$, so $\mathsf m_0(b) \le |X_i| = m$.  By the minimality of $m$, we have $\icone(X') \subset \icone(\gcd(X')) \subsetneq \ZZ_{\ge 0}$ for any subset $X' \subset X$ with $|X'| < m$.  In particular, the set
$$\icone(X) \setminus \bigcup_{|X'| < m} \icone(X')$$
has infinite cardinality, and is precisely the set of elements $b \in \icone(X)$ with $\mathsf m_0(b) \ge m$.  The second claim now follows from the first.
\end{proof}

\begin{example}\label{e:eventualboundm0knap}
Return to $X = \{4,6,15\}$ as in Example~\ref{e:eventualm0knap}.  Proposition~\ref{p:eventualboundm0knap} ensures $\mathsf m(b) \le 2$ for all $b \ge 42$, even though $M_0(X) = 3$ is achieved at four distinct values prior to $b = 42$.  Algebraically, this is because $\gcd(4,15) = 1$, so every $b \ge 42$ lies in $\icone(4,15)$ and thus satisfies $\mathsf m_0(b) \le 2$.  However, $\mathsf m_0(25) = 3$, since $(1,1,1)^T$ is the only solution in $\mathsf P_X(25)$.
\end{example}

\subsection*{Acknowledgements}%

The second author acknowledges support from NSF DMS-grant 1522158.
We are grateful to Lenny Fushansky, Martin Henk, Fr\'ed\'eric Meunier, and Daniele Micciancio
for their very useful comments and references. We would also like to thank the referees for their valuable
suggestions.


\end{document}